\documentclass{amsart}
\usepackage{a4}
\usepackage{graphics}
\usepackage{color}
\usepackage{amssymb}
\usepackage[all]{xy}
\usepackage{amsmath}
\usepackage{stmaryrd}
\usepackage{longtable}
\usepackage{setspace}
\usepackage{booktabs}

\CompileMatrices

\newtheorem{theorem}{Theorem}[section]
\newtheorem{lemma}[theorem]{Lemma}
\newtheorem{proposition}[theorem]{Proposition}
\newtheorem{corollary}[theorem]{Corollary}

\theoremstyle{definition}

\newtheorem{example}[theorem]{Example}
\newtheorem{remark}[theorem]{Remark}

\newtheorem{construction}[theorem]{Construction}
\newtheorem*{acknowledgements}{Acknowledgements}

\numberwithin{equation}{section}
\def\bangle#1{\langle #1 \rangle}

\def\KK{{\mathbb K}}
\def\ZZ{{\mathbb Z}}
\def\NN{{\mathbb N}}
\def\QQ{{\mathbb Q}}

\def\Cl{\operatorname{Cl}}
\def\Pic{\operatorname{Pic}}
\def\Spec{{\rm Spec}}

\def\Aut{{\rm Aut}}

\def\gcd{\mathrm{gcd}}
\def\lcm{\mathrm{lcm}}

\begin{document}
\title[On Fano varieties with torus action of complexity one]%
{On Fano varieties with torus action of complexity one}
\author[E.~Herppich]{Elaine Herppich} 
\address{Mathematisches Institut, Universit\"at T\"ubingen,
Auf der Morgenstelle 10, 72076 T\"ubingen, Germany}
\email{elaine.herppich@uni-tuebingen.de}

\begin{abstract}
In this work we provide effective bounds
and classification results for 
rational $\QQ$-factorial Fano varieties
with a complexity-one torus action
and Picard number one depending on the 
invariants dimension and Picard index.
This complements earlier work
by Hausen, S\" u{\ss} and the author,
where the case of free divisor class group
of rank one was treated.
\end{abstract}


\maketitle

\section{Statement of the results}

The subject of this paper are rational $\QQ$-factorial
Fano varieties $X$ defined over an algebraically closed
field $\KK$ of characteristic zero, see for example 
\cite{Is} and \cite{MoMu} for classical work. A more
recent focus in this field are toric Fano varieties,
where one uses the description in terms of lattice
polytopes (\cite{Ba}, \cite{KaKrNi}, \cite{Ka2}, etc.).
Here we study
the case that $X$ comes, more generally,
with an effective 
action of a torus $T$ of complexity one,
i.e. $\dim X-\dim T=1$; 
by Fano varieties we mean normal projective varieties
with ample anticanonical divisor $-K_X$.
We continue the work of \cite{HaHeSu}
where classification results for
the case $\Cl(X)=\ZZ$ were given.
{In this paper we study the
more general case of Picard number one,
i.e.~we allow torsion in the divisor class group.
A first step is Theorem \ref{Th:FiniteIndex} 
where we provide effective bounds for 
the number of deformation types 
of Fano varieties $X$ as above
with fixed dimension $d$ and Picard index 
$\mu:=[\Cl(X):\Pic(X)]$.
As a consequence we obtain restricting
statements about the number $\delta(d,\mu)$ 
of different deformation types of 
$\QQ$-factorial $d$-dimensional Fano varieties with a 
complexity-one torus action, Picard number one and
Picard index~$\mu$.
In the toric situation $\delta(d,\mu)$ is bounded above by  
$\mu^{d^2}$.
For the non-toric case we get the following asymptotic results:

\begin{theorem}\label{Th:asymptotic}
For fixed $d\in\ZZ_{>0}$,
the number $\delta(d,\mu)$ is asymptotically 
bounded above by $\mu^{(1+\varepsilon)\mu^2}$
for $\varepsilon>0$ arbitrarily small,
and for fixed $\mu\in\ZZ_{>0}$,
it is asymptotically 
bounded above  by $d^{Ad}$ 
with a constant $A$ depending only on $\mu$.
\end{theorem}

We turn to the classification. 
Our approach 
uses the Cox ring $\mathcal R(X)$
which is defined by 
$$
\mathcal R(X)
\ = \
\bigoplus_{D\in\Cl(X)}\Gamma(X,\mathcal O_X(D)).
$$
Given this ring, the variety $X$ can be realized
as a quotient of an open subset in $\mathrm{Spec}(\mathcal R(X))$
by the action of a diagonalizable group.

According to \cite[Theorem 1.3]{HaSu}
the Cox ring of a normal complete
rational variety with a complexity-one torus action 
is finitely generated.
Furthermore, every such Fano variety is uniquely determined 
by its Cox ring (as a $\Cl(X)$-graded ring).
In case of Picard number one
the toric varieties of this type 
correspond to the fake weighted projective spaces 
as defined in \cite{Ka} and the Cox ring is polynomial.
In the subsequent theorems we list 
non-toric complexity-one Fanos with Picard number one
in the  cases where $\Cl(X)$ has non trivial torsion;
for the non-toric results in case of $\Cl(X)=\ZZ$ 
we refer to \cite{HaHeSu}.
The Cox rings are described in terms of generators and relations and 
we specify  the $\Cl(X)$-grading by giving the degrees of the generators.
Additionally we list the degree of the Fano varieties $d_X:=(-K_X)^d$
and the Gorenstein index $\iota(X)$, i.e. the smallest 
positive integer such that $\iota(X)\cdot K_X$ is Cartier.

\begin{theorem}\label{Theo:surfaces<=6}
Let $X$ be a non-toric Fano surface
with an effective $\KK^*$-action, Picard
number one, non trivial torsion in the class group
and $[\Cl(X):\Pic(X)]\leq 6$.
Then its Cox ring is precisely one of the following.

\begin{center}
\begin{longtable}[htbp]{lllccc}
\multicolumn{6}{c}{\bf $[\Cl(X):\Pic(X)] = 2$}
\\[1ex]
\toprule
No.
&
$\mathcal{R}(X)$
&
$\Cl(X)$
&
\emph{grading} 
&
$d_X$
& 
$\iota(X)$
\\
\midrule
1
&
$\KK[{T_1,\ldots,T_4}]/ \bangle{T_1T_2^3+T_3^4+T_4^2}$
&
$\ZZ\oplus \ZZ/2\ZZ$
&
$\left(
\begin{smallmatrix}
1 & 1 & 1 & 2\\
\overline 0 & \overline 0 & \overline 1 & \overline 1
\end{smallmatrix}
\right)$
&
$1$
&
$1$
\\
\bottomrule
\\[2ex]
\\
\\
\\
\multicolumn{6}{c}{\bf $[\Cl(X):\Pic(X)] = 3$}
\\[1ex]
\toprule
No.
&
$\mathcal{R}(X)$
&
$\Cl(X)$
&
\emph{grading} 
&
$d_X$
&
$\iota(X)$
\\
\midrule
2
&
$\KK[{T_1,\ldots,T_4}]/ \bangle{T_1T_2^2+T_3^3+T_4^3}$
&
$\ZZ\oplus\ZZ/3\ZZ$
&
$
\left(
\begin{smallmatrix}
1 & 1 & 1 & 1\\
\overline 1 & \overline 1 & \overline 2 & \overline 0
\end{smallmatrix}
\right)$
&
$1$
&
$1$
\\
\bottomrule
\\[2ex]
\multicolumn{6}{c}{\bf $[\Cl(X):\Pic(X)] = 4$}
\\[1ex]
\toprule
No.
&
$\mathcal{R}(X)$
&
$\Cl(X)$
&
\emph{grading} 
&
$d_X$
&
$\iota(X)$
\\
\midrule
3
&
$\KK[{T_1,T_2,T_3,S_1}] / \bangle{T_1^2+T_2^2+T_3^2}$
&
$\ZZ\oplus\ZZ/2\ZZ\oplus\ZZ/2\ZZ$
&
$
\left(
\begin{smallmatrix}
1 & 1 & 1 & 1\\
\overline 1 & \overline 1 & \overline 0 & \overline 0\\
\overline 0 & \overline 1 & \overline 1 & \overline 0
\end{smallmatrix}
\right)$
&
$2$
&
$1$
\\
\midrule
4
&
$\KK[{T_1,\ldots,T_4}] / \bangle{T_1T_2+T_3^2+T_4^2}$
&
$\ZZ\oplus\ZZ/4\ZZ$
&
$
\left(
\begin{smallmatrix}
1 & 1 & 1 & 1\\
\overline 1 & \overline 3 & \overline 2 & \overline 0
\end{smallmatrix}
\right)$
&
$2$
&
$1$
\\
\midrule
5
&
$\KK[{T_1,\ldots,T_4}] / \bangle{T_1^2T_2+T_3^2+T_4^4}$
&
$\ZZ\oplus\ZZ/2\ZZ$
&
$
\left(
\begin{smallmatrix}
1 & 2 & 2 & 1\\
\overline 1 & \overline 0 & \overline 1 & \overline 0
\end{smallmatrix}
\right)$
&
$2$
&
$1$
\\
\midrule
6
&
$\KK[{T_1,\ldots,T_4}] / \bangle{T_1T_2^2+T_3^6+T_4^2}$
&
$\ZZ\oplus\ZZ/2\ZZ$
&
$
\left(
\begin{smallmatrix}
2 & 2 & 1 & 3\\
\overline 0 & \overline 1 & \overline 0 & \overline 1
\end{smallmatrix}
\right)$
&
$1$
&
$2$
\\
\midrule
7
&
$\KK[{T_1,\ldots,T_5}] / 
\bangle{
\begin{smallmatrix}
T_1T_2+T_3^2+T_4^2,\\
\lambda T_3^2+T_4^2+T_5^2
\end{smallmatrix}}$
&
$\ZZ\oplus\ZZ/2\ZZ\oplus\ZZ/2\ZZ$
&
$
\left(
\begin{smallmatrix}
1 & 1 & 1 & 1 & 1\\
\overline 1 & \overline 1 & \overline 0 & \overline 1 &\overline 0\\
\overline 0 & \overline 0 & \overline 1 & \overline 1 &\overline 0
\end{smallmatrix}
\right)$
&
$1$
&
$1$
\\
\bottomrule
\\[2ex]
\multicolumn{6}{c}{\bf $[\Cl(X):\Pic(X)] = 6$}
\\[1ex]
\toprule
No.
&
$\mathcal{R}(X)$
&
$\Cl(X)$
&
\emph{grading} 
&
$d_X$
&
$\iota(X)$
\\
\midrule
8
&
$\KK[{T_1,T_2,T_3,S_1}] / \bangle{T_1^3+T_2^3+T_3^2}$
&
$\ZZ\oplus\ZZ/3\ZZ$
&
$\left(
\begin{smallmatrix}
2 & 2 & 3 & 1\\
\overline 1 & \overline 2 & \overline 0 & \overline 0
\end{smallmatrix}
\right)
$
&
$2/3$
&
$3$
\\
\midrule
9
&
$\KK[{T_1,\ldots,T_4}]/ \bangle{T_1T_2+T_3^3+T_4^3}$
&
$\ZZ\oplus\ZZ/3\ZZ$
&
$
\left(
\begin{smallmatrix}
1 & 2 & 1 & 1\\
\overline 1 & \overline 2 & \overline 2 & \overline 0
\end{smallmatrix}
\right)$
&
$2$
&
$1$
\\
\midrule
10
&
$\KK[{T_1,\ldots,T_4}] / \bangle{T_1T_2+T_3^2+T_4^4}$
&
$\ZZ\oplus\ZZ/2\ZZ$
&
$
\left(
\begin{smallmatrix}
3 & 1 & 2 & 1\\
\overline 1 & \overline 1 & \overline 1 & \overline 0
\end{smallmatrix}
\right)$
&
$3$
&
$1$
\\
\midrule
11
&
$\KK[{T_1,\ldots,T_4}] / \bangle{T_1T_2^5+T_3^2+T_4^8}$
&
$\ZZ\oplus\ZZ/2\ZZ$
&
$
\left(
\begin{smallmatrix}
3 & 1 & 4 & 1\\
\overline 1 & \overline 1 & \overline 1 & \overline 0
\end{smallmatrix}
\right)$
&
$1/3$
&
$3$
\\
\bottomrule
\end{longtable}
\end{center}
where the parameter $\lambda$ occuring in the second
relation of surface number $7$ can be any element
of $\KK^*\setminus \{1\}$.
Furthermore the Cox rings listed above are 
pairwise non-isomorphic as graded rings.
\end{theorem}

\begin{remark}
Gorenstein surfaces are well known to have ADE-singularities
which are in particular canonical.
Consequently the surfaces of number $1$ to $5$, $7$, $9$ and $10$
are canonical.
Furthermore in \cite{Su} all log-terminal 
Del Pezzo $\KK^*$-surfaces 
of Gorenstein index up to three are classified.
Comparing the surfaces listed in
\cite[Theorems~4.9, 4.10]{Su} with the table above 
shows that number $11$ is  not log-terminal.
The resolution of this surface can 
be explicitly computed by using the 
method of toric ambient modification
as demonstrated in \cite[Examples~3.20, 3.21]{Ha3}.
\end{remark}

\begin{theorem}
\label{thm:3fano2}
Let $X$ be a three-dimensional non-toric Fano 
variety
with an effective two-torus action,
Picard number one, non trivial torsion in the class group  
and $[\Cl(X):\Pic(X)]=2$.
Then its Cox ring is precisely one of the 
following.
\begin{center}
\begin{longtable}[htbp]{lllcccc}
\toprule
No. 
&
$\mathcal{R}(X)$ 
&
$\Cl(X)$
&
\emph{grading} 
& 
$d_X$
&
$\iota(X)$
\\
\midrule 1 &
$\KK[{T_1,\ldots,T_4,S_1}]/
 \langle   T_1T_2+T_3^2+T_4^2  \rangle$
&
$\ZZ\oplus\ZZ/2\ZZ$
&
 $\left(
\begin{smallmatrix}
1 & 1 & 1 & 1 & 1\\
\overline 1 &\overline 1 &\overline 1 &\overline 0 &\overline 0 
\end{smallmatrix}
\right)$
&
$27$
&
$1$
\\
\midrule 2 &
$\KK[{T_1,\ldots,T_4,S_1}]/
 \langle   T_1T_2^3+T_3^2+T_4^4  \rangle$
&
$\ZZ\oplus\ZZ/2\ZZ$
&
 $\left(
\begin{smallmatrix}
1 & 1 & 2 & 1 & 1\\
\overline 0 &\overline 0 &\overline 1 &\overline 0 &\overline 1 
\end{smallmatrix}
\right)$
&
$8$
&
$2$
\\
\midrule 3 &
$\KK[{T_1,\ldots,T_4,S_1}]/
 \langle   T_1T_2^3+T_3^2+T_4^4  \rangle$
&
$\ZZ\oplus\ZZ/2\ZZ$
&
 $\left(
\begin{smallmatrix}
1 & 1 & 2 & 1 & 1\\
\overline 0 &\overline 0 &\overline 1 &\overline 1 &\overline 0 
\end{smallmatrix}
\right)$
&
$8$
&
$1$
\\
\midrule 4 &
$\KK[{T_1,\ldots,T_4,S_1}]/
 \langle   T_1T_2^3+T_3^2+T_4^4  \rangle$
&
$\ZZ\oplus\ZZ/2\ZZ$
&
 $\left(
\begin{smallmatrix}
1 & 1 & 2 & 1 & 1\\
\overline 0 &\overline 0 &\overline 1 &\overline 1 &\overline 1 
\end{smallmatrix}
\right)$
&
$8$
&
$2$
\\
\midrule 5 &
$\KK[{T_1,\ldots,T_4,S_1}]/
 \langle   T_1T_2^5+T_3^2+T_4^6  \rangle$
&
$\ZZ\oplus\ZZ/2\ZZ$
&
 $\left(
\begin{smallmatrix}
1 & 1 & 3 & 1 & 1\\
\overline 0 &\overline 0 &\overline 0 &\overline 1 &\overline 1 
\end{smallmatrix}
\right)$
&
$1$
&
$1$
\\
\midrule 6 &
$\KK[{T_1,\ldots,T_4,S_1}]/
 \langle   T_1T_2^5+T_3^2+T_4^6  \rangle$
&
$\ZZ\oplus\ZZ/2\ZZ$
&
 $\left(
\begin{smallmatrix}
1 & 1 & 3 & 1 & 1\\
\overline 0 &\overline 0 &\overline 1 &\overline 0 &\overline 1 
\end{smallmatrix}
\right)$
&
$1$
&
$1$
\\
\midrule 7 &
$\KK[{T_1,\ldots,T_4,S_1}]/
 \langle   T_1^2T_2^4+T_3^2+T_4^3  \rangle$
&
$\ZZ\oplus\ZZ/2\ZZ$
&
 $\left(
\begin{smallmatrix}
1 & 1 & 3 & 2 & 1\\
\overline 0 &\overline 0 &\overline 1 &\overline 0 &\overline 1
\end{smallmatrix}
\right)$
&
$4$
&
$1$
\\
\midrule 8 &
$\KK[{T_1,\ldots,T_4,S_1}]/
 \langle   T_1^2T_2^4+T_3^2+T_4^3  \rangle$
&
$\ZZ\oplus\ZZ/2\ZZ$
&
 $\left(
\begin{smallmatrix}
1 & 1 & 3 & 2 & 1\\
\overline 0 &\overline 1 &\overline 1 &\overline 0 &\overline 0 
\end{smallmatrix}
\right)$
&
$4$
&
$1$
\\
\midrule 9 &
$\KK[{T_1,\ldots,T_4,S_1}]/
 \langle   T_1T_2^5+T_3^3+T_4^2  \rangle$
&
$\ZZ\oplus\ZZ/2\ZZ$
&
 $\left(
\begin{smallmatrix}
1 & 1 & 2 & 3 & 1\\
\overline 1 &\overline 1 &\overline 0 &\overline 0 &\overline 0 
\end{smallmatrix}
\right)$
&
$4$
&
$2$
\\
\midrule 10 &
$\KK[{T_1,\ldots,T_5}]/
 \langle   T_1T_2^3+T_3^2T_4^2+T_5^2  \rangle$
&
$\ZZ\oplus\ZZ/2\ZZ$
&
 $\left(
\begin{smallmatrix}
1 & 1 & 1 & 1 & 2\\
\overline 0 &\overline 0 &\overline 1 &\overline 1 &\overline 1 
\end{smallmatrix}
\right)$
&
$8$
&
$2$
\\
\midrule 11 &
$\KK[{T_1,\ldots,T_5}]/
 \langle   T_1T_2^3+T_3^2T_4^2+T_5^4  \rangle$
&
$\ZZ\oplus\ZZ/2\ZZ$
&
 $\left(
\begin{smallmatrix}
1 & 1 & 1 & 1 & 1\\
\overline 1 &\overline 1 &\overline 0 &\overline 1 &\overline 0 
\end{smallmatrix}
\right)$
&
$2$
&
$1$
\\
\midrule 12 &
$\KK[{T_1,\ldots,T_5}]/
 \langle   T_1T_2^5+T_3^2T_4^4+T_5^2  \rangle$
&
$\ZZ\oplus\ZZ/2\ZZ$
&
 $\left(
\begin{smallmatrix}
1 & 1 & 1 & 1 & 3\\
\overline 0 &\overline 0 &\overline 1 &\overline 1 &\overline 0 
\end{smallmatrix}
\right)$
&
$1$
&
$1$
\\
\midrule 13 &
$\KK[{T_1,\ldots,T_5}]/
 \langle   T_1T_2^5+T_3^2T_4^4+T_5^2  \rangle$
&
$\ZZ\oplus\ZZ/2\ZZ$
&
 $\left(
\begin{smallmatrix}
1 & 1 & 1 & 1 & 3\\
\overline 0 &\overline 0 &\overline 0 &\overline 1 &\overline 1 
\end{smallmatrix}
\right)$
&
$1$
&
$1$
\\
\midrule 14 &
$\KK[{T_1,\ldots,T_5}]/
 \langle   T_1^2T_2^4+T_3^3T_4^3+T_5^2  \rangle$
&
$\ZZ\oplus\ZZ/2\ZZ$
&
 $\left(
\begin{smallmatrix}
1 & 1 & 1 & 1 & 3\\
\overline 0 &\overline 0 &\overline 1 &\overline 1 &\overline 1 
\end{smallmatrix}
\right)$
&
$1$
&
$2$
\\
\midrule 15 &
$\KK[{T_1,\ldots,T_5}]/
 \langle   T_1^2T_2^4+T_3^3T_4^3+T_5^2  \rangle$
&
$\ZZ\oplus\ZZ/2\ZZ$
&
 $\left(
\begin{smallmatrix}
1 & 1 & 1 & 1 & 3\\
\overline 0 &\overline 1 &\overline 0 &\overline 0 &\overline 1 
\end{smallmatrix}
\right)$
&
$1$
&
$2$
\\
\midrule 16 &
$\KK[{T_1,\ldots,T_5}]/
 \langle   T_1^2T_2^4+T_3^5T_4+T_5^2  \rangle$
&
$\ZZ\oplus\ZZ/2\ZZ$
&
 $\left(
\begin{smallmatrix}
1 & 1 & 1 & 1 & 3\\
\overline 0 &\overline 0 &\overline 1 &\overline 1 &\overline 1 
\end{smallmatrix}
\right)$
&
$1$
&
$2$
\\
\midrule 17 &
$\KK[{T_1,\ldots,T_5}]/
 \langle   T_1^2T_2^4+T_3^5T_4+T_5^2  \rangle$
&
$\ZZ\oplus\ZZ/2\ZZ$
&
 $\left(
\begin{smallmatrix}
1 & 1 & 1 & 1 & 3\\
\overline 0 &\overline 1 &\overline 0 &\overline 0 &\overline 1
\end{smallmatrix}
\right)$
&
$1$
&
$2$
\\
\midrule 18 &
$\KK[{T_1,\ldots,T_5}]/
 \langle   T_1T_2+T_3T_4+T_5^2  \rangle$
&
$\ZZ\oplus\ZZ/2\ZZ$
&
 $\left(
\begin{smallmatrix}
1 & 1 & 1 & 1 & 1\\
\overline 0 &\overline 0 &\overline 1 &\overline 1 &\overline 1 
\end{smallmatrix}
\right)$
&
$27$
&
$2$
\\
\midrule 19 &
$\KK[{T_1,\ldots,T_5}]/
 \langle   T_1T_2+T_3T_4+T_5^2  \rangle$
&
$\ZZ\oplus\ZZ/2\ZZ$
&
 $\left(
\begin{smallmatrix}
1 & 1 & 1 & 1 & 1\\
\overline 0 &\overline 0 &\overline 1 &\overline 1 &\overline 0 
\end{smallmatrix}
\right)$
&
$27$
&
$1$
\\
\midrule 20 &
$\KK[{T_1,\ldots,T_5}]/
 \langle   T_1T_2^2+T_3T_4^2+T_5^3  \rangle$
&
$\ZZ\oplus\ZZ/2\ZZ$
&
 $\left(
\begin{smallmatrix}
1 & 1 & 1 & 1 & 1\\
\overline 0 &\overline 1 &\overline 0 &\overline 1 &\overline 0
\end{smallmatrix}
\right)$
&
$12$
&
$1$
\\
\midrule 21 &
$\KK[{T_1,\ldots,T_5}]/
 \langle   T_1T_2^3+T_3T_4^3+T_5^4  \rangle$
&
$\ZZ\oplus\ZZ/2\ZZ$
&
 $\left(
\begin{smallmatrix}
1 & 1 & 1 & 1 & 1\\
\overline 1 &\overline 1 &\overline 0 &\overline 0 &\overline 0 
\end{smallmatrix}
\right)$
&
$2$
&
$2$
\\
\midrule 22 &
$\KK[{T_1,\ldots,T_5}]/
 \langle   T_1T_2^3+T_3T_4^3+T_5^4  \rangle$
&
$\ZZ\oplus\ZZ/2\ZZ$
&
 $\left(
\begin{smallmatrix}
1 & 1 & 1 & 1 & 1\\
\overline 0 &\overline 0 &\overline 1 &\overline 1 &\overline 0 
\end{smallmatrix}
\right)$
&
$2$
&
$1$
\\
\midrule 23 &
$\KK[{T_1,\ldots,T_5}]/
 \langle   T_1T_2^3+T_3T_4^3+T_5^2  \rangle$
&
$\ZZ\oplus\ZZ/2\ZZ$
&
 $\left(
\begin{smallmatrix}
1 & 1 & 1 & 1 & 2\\
\overline 0 &\overline 0 &\overline 1 &\overline 1 &\overline 0
\end{smallmatrix}
\right)$
&
$8$
&
$2$
\\
\midrule 24 &
$\KK[{T_1,\ldots,T_5}]/
 \langle   T_1T_2^3+T_3T_4^3+T_5^2  \rangle$
&
$\ZZ\oplus\ZZ/2\ZZ$
&
 $\left(
\begin{smallmatrix}
1 & 1 & 1 & 1 & 2\\
\overline 0 &\overline 0 &\overline 1 &\overline 1 &\overline 1
\end{smallmatrix}
\right)$
&
$8$
&
$2$
\\
\midrule 25 &
$\KK[{T_1,\ldots,T_5}]/
 \langle   T_1T_2^5+T_3T_4^5+T_5^2  \rangle$
&
$\ZZ\oplus\ZZ/2\ZZ$
&
 $\left(
\begin{smallmatrix}
1 & 1 & 1 & 1 & 3\\
\overline 0 &\overline 0 &\overline 1 &\overline 1 &\overline 0
\end{smallmatrix}
\right)$
&
$1$
&
$1$
\\
\midrule 26 &
$\KK[{T_1,\ldots,T_5}]/
 \langle   T_1T_2^5+T_3T_4^5+T_5^2  \rangle$
&
$\ZZ\oplus\ZZ/2\ZZ$
&
 $\left(
\begin{smallmatrix}
1 & 1 & 1 & 1 & 3\\
\overline 0 &\overline 0 &\overline 1 &\overline 1 &\overline 1
\end{smallmatrix}
\right)$
&
$1$
&
$2$
\\
\midrule 27 &
$\KK[{T_1,\ldots,T_5}]/
 \langle   T_1T_2^5+T_3^3T_4^3+T_5^2  \rangle$
&
$\ZZ\oplus\ZZ/2\ZZ$
&
 $\left(
\begin{smallmatrix}
1 & 1 & 1 & 1 & 3\\
\overline 0 &\overline 0 &\overline 1 &\overline 1 &\overline 0
\end{smallmatrix}
\right)$
&
$1$
&
$1$
\\
\midrule 28 &
$\KK[{T_1,\ldots,T_5}]/
 \langle   T_1T_2^5+T_3^3T_4^3+T_5^2  \rangle$
&
$\ZZ\oplus\ZZ/2\ZZ$
&
 $\left(
\begin{smallmatrix}
1 & 1 & 1 & 1 & 3\\
\overline 0 &\overline 0 &\overline 1 &\overline 1 &\overline 1
\end{smallmatrix}
\right)$
&
$1$
&
$2$
\\
\midrule 29 &
$\KK[{T_1,\ldots,T_5}]/
 \langle   T_1T_2T_3^2+T_4^2+T_5^4  \rangle$
&
$\ZZ\oplus\ZZ/2\ZZ$
&
 $\left(
\begin{smallmatrix}
1 & 1 & 1 & 2 & 1\\
\overline 0 &\overline 0 &\overline 1 &\overline 1 &\overline 0 
\end{smallmatrix}
\right)$
&
$8$
&
$2$
\\
\midrule 30 &
$\KK[{T_1,\ldots,T_5}]/
 \langle   T_1T_2T_3^2+T_4^2+T_5^4  \rangle$
&
$\ZZ\oplus\ZZ/2\ZZ$
&
 $\left(
\begin{smallmatrix}
1 & 1 & 1 & 2 & 1\\
\overline 0 &\overline 0 &\overline 0 &\overline 1 &\overline 1 
\end{smallmatrix}
\right)$
&
$8$
&
$1$
\\
\midrule 31 &
$\KK[{T_1,\ldots,T_5}]/
 \langle   T_1T_2T_3^2+T_4^2+T_5^4  \rangle$
&
$\ZZ\oplus\ZZ/2\ZZ$
&
 $\left(
\begin{smallmatrix}
1 & 1 & 1 & 2 & 1\\
\overline 0 &\overline 0 &\overline 1 &\overline 1 &\overline 1 
\end{smallmatrix}
\right)$
&
$8$
&
$1$
\\
\midrule 32 &
$\KK[{T_1,\ldots,T_5}]/
 \langle   T_1T_2T_3^4+T_4^2+T_5^6  \rangle$
&
$\ZZ\oplus\ZZ/2\ZZ$
&
 $\left(
\begin{smallmatrix}
1 & 1 & 1 & 3 & 1\\
\overline 0 &\overline 0 &\overline 1 &\overline 1 &\overline 0 
\end{smallmatrix}
\right)$
&
$1$
&
$1$
\\
\midrule 33 &
$\KK[{T_1,\ldots,T_5}]/
 \langle   T_1T_2T_3^4+T_4^2+T_5^6  \rangle$
&
$\ZZ\oplus\ZZ/2\ZZ$
&
 $\left(
\begin{smallmatrix}
1 & 1 & 1 & 3 & 1\\
\overline 0 &\overline 0 &\overline 1 &\overline 0 &\overline 1 
\end{smallmatrix}
\right)$
&
$1$
&
$1$
\\
\midrule 34 &
$\KK[{T_1,\ldots,T_5}]/
 \langle   T_1T_2^2T_3^3+T_4^2+T_5^6  \rangle$
&
$\ZZ\oplus\ZZ/2\ZZ$
&
 $\left(
\begin{smallmatrix}
1 & 1 & 1 & 3 & 1\\
\overline 0 &\overline 1 &\overline 0 &\overline 1 &\overline 0 
\end{smallmatrix}
\right)$
&
$1$
&
$1$
\\
\midrule 35 &
$\KK[{T_1,\ldots,T_5}]/
 \langle   T_1T_2^2T_3^3+T_4^2+T_5^6  \rangle$
&
$\ZZ\oplus\ZZ/2\ZZ$
&
 $\left(
\begin{smallmatrix}
1 & 1 & 1 & 3 & 1\\
\overline 0 &\overline 1 &\overline 0 &\overline 0 &\overline 1 
\end{smallmatrix}
\right)$
&
$1$
&
$1$
\\
\midrule 36 &
$\KK[{T_1,\ldots,T_5}]/
 \langle   T_1T_2^2T_3^3+T_4^3+T_5^2  \rangle$
&
$\ZZ\oplus\ZZ/2\ZZ$
&
 $\left(
\begin{smallmatrix}
1 & 1 & 1 & 2 & 3\\
\overline 1 &\overline 0 &\overline 1 &\overline 0 &\overline 0 
\end{smallmatrix}
\right)$
&
$4$
&
$1$
\\
\midrule 37 &
$\KK[{T_1,\ldots,T_5}]/
 \langle   T_1T_2T_3^4+T_4^3+T_5^2  \rangle$
&
$\ZZ\oplus\ZZ/2\ZZ$
&
 $\left(
\begin{smallmatrix}
1 & 1 & 1 & 2 & 3\\
\overline 0 &\overline 0 &\overline 1 &\overline 0 &\overline 1 
\end{smallmatrix}
\right)$
&
$4$
&
$2$
\\
\midrule
38 
&
$\KK[{T_1,\ldots,T_6}] / 
\bangle{
\begin{smallmatrix}
T_1T_2+T_3T_4+T_5^2,\\
\lambda T_3T_4+T_5^2+T_6^2
\end{smallmatrix}}$
&
$\ZZ\oplus\ZZ/2\ZZ$
&
$
\left(
\begin{smallmatrix}
1 & 1 & 1 & 1 & 1 & 1\\
\overline 0 & \overline 0 & \overline 1 & \overline 1 &\overline 1 & \overline 0\\
\end{smallmatrix}
\right)$
&
$16$
&
$2$
\\
\bottomrule
\end{longtable}
\end{center}
where the parameter $\lambda$ occuring in the second
relation of threefold number $38$ can be any element
of $\KK^*\setminus \{1\}$.
Furthermore the Cox rings listed above are 
pairwise non-isomorphic as graded rings.
\end{theorem}

Let $X$ be a normal complete rational variety 
coming with a complexity-one torus action of $T$.
Consider the $T$-invariant open subset $X_0$
consisting of all points $x\in X$ having finite isotropy group.
According to \cite[Corollary~3]{Sum} there is a geometric quotient 
$q\colon X_0\to X_0/T$ such that $X_0/T$ is 
irreducible and normal but possibly not separated.
The property of the orbit space $X_0/T$ 
being separated is reflected in the Cox ring relations
by the condition that each monomial depends on only one variable,
e.g. surface number $3$
in Theorem \ref{Theo:surfaces<=6};
see \cite[Theorem 1.2]{HaSu}.
Geometrically, this means, 
that every orbit is contained in the closure
of either exactly one maximal orbit 
or of infinitely many maximal orbits.
For such varieties we have the following general finiteness statement:

\begin{theorem}\label{Theo:sepquot}
The number of $d$-dimensional normal complete 
rational varieties of Picard number one
with a complexity-one torus action of $T$ 
and Picard index $\mu$, such that $X_0/T$ is separated, is finite.
\end{theorem}

\section{Description of the Cox ring}

We briefly recall from \cite{HaHe}
a construction of $\QQ$-factorial normal rational
projective varieties with a 
complexity-one torus action.
Here, we specialize to the case of Picard number one;
the details are given in \cite[Proposition~2.4]{HaHe}.

\begin{construction}\label{Con:R(A,n,L,m)}\label{Con:varX}
For $r \ge 1$, consider a sequence
$A = (a_0, \ldots, a_r)$ 
of pairwise linearly independent 
vectors in $\KK^2$, 
a sequence $\mathfrak{n} = (n_0, \ldots, n_r)$ 
of positive integers,
a non-negative integer $m$
and a family  $L = (l_{ij})$
of positive integers,
where $0 \le i \le r$ and 
$1 \le j \le n_i$.
Set 
$$
R(A,\mathfrak n,L,m)
\ := \
\KK[T_{ij},S_k] 
\ / \
\bangle{g_0, \ldots, g_{r-2}}.
$$ 
where the $T_{ij}$ are indexed by
$0 \le i \le r, \; 1 \le j \le n_i$,
the $S_k$ by $1\leq k\leq m$ and 
the relations $g_i$ are defined as follows:
Set $T_i^{l_i}:=T_{i1}^{l_{i1}}\cdots T_{in_i}^{l_{in_i}}$
and
$$
g_i
\ := \
\det
\begin{pmatrix}
a_i & a_{i+1} & a_{i+2}\\
T_{i}^{l_i}  &T_{i+1}^{l_{i+1}}  &T_{i+2}^{l_{i+2}}  
\end{pmatrix}.
$$
Define $n:=n_0+\ldots+n_r$ and 
let $K:=\ZZ\oplus K^t$
be an abelian group with
torsion part $K^t$.
Suppose that $R(A,\mathfrak n, L,m)$ is 
positively $K$-graded
via
$$
\deg\, T_{ij}
\ = \
 w_{ij}\in K,
\qquad
\deg\, S_k
\ = \
u_k\in K,
$$
i.e. $w_{ij},u_k\in\ZZ_{\geq 0}\otimes K^t$,
and that any
$n+m-1$ of these degrees generate $K$ as a group.
The $K$-grading defines a diagonal
action of  $H:=\Spec\; \KK[K]$ on $\KK^{n+m}$.
By construction
$$
\overline X
\ := \
V(g_{i};0\leq i\leq r-2)
\ = \
\Spec\, R(A,\mathfrak n,L,m)
$$
is invariant under this $H$-action.
The open set $\KK^{n+m}\setminus\{0\}$ allows
a geometric quotient of this $H$-action
which is denoted by
$p\colon \KK^{n+m}\setminus\{0\} \to Z$,
where the toric variety $Z$ is a 
a fake weighted projective space.
Furthermore we get a geometric quotient 
$p\colon\widehat X\to X$ of the
embedded open subset
$\widehat X := \overline X\setminus \{0\}$.
\[
\begin{xy}
\xymatrix{
& {\widehat X}\; \ar[d]^{ p} \ar@{^ {(}->}[r] 
&\KK^{n+m}\setminus\{0\}\ar[d]^{p}  \\
&X\; \ar@{^ {(}->} [r] &Z}
\end{xy}
\]
The quotient space $X:=\widehat X\sslash H$ 
is a $\QQ$-factorial normal projective variety
of dimension 
$$
\dim(X)
\ = \
n+m-r.
$$ 
It has divisor class group $\Cl(X)=K$, Cox ring 
$\mathcal R(X) =R(A,\mathfrak n,L,m)$
complexity-one torus action.
This torus is 
given by the stabilizer of $X$ under the action of 
the maximal torus $T_Z$ of $Z$.\\
Note that, if there is an index $0\leq i\leq r$
such that $l_{i1}=1$ and $n_i=1$, then there is
at least one relation containing a linear term.
In this case the ring is isomorphic to the 
polynomial ring that we get, if we omit the relations of this type.
Consequently we may always assume $l_{i1}n_{i}\neq 1$.
\end{construction}

\begin{remark}
Varieties with complexity-one action,
as constructed in \ref{Con:varX}, can be considered
as a generalized version of well-formed complete
intersections in weighted projective spaces,
in the sense of \cite{IaFl}.
\end{remark}

According to \cite[Theorem~1.5]{HaHe}
every  $\QQ$-factorial normal complete rational variety  $X$ 
with a complexity-one torus action and Picard number one
has a  Cox ring $R(X)$
which is isomorphic as a graded ring to 
some $K$-graded algebra
$R(A,\mathfrak n,L,m)$ with $K\cong\Cl(X)$.

We collect some geometric properties
of the varieties  $X$ just constructed.
 Every element 
$w\in K=\ZZ\oplus K^t$ 
can be written as $w=w^0+w^t$ where 
$w^0\in \ZZ$ and $w^t\in K^t$.
Furthermore, every 
$\overline x=(\overline x_{ij},\overline x_k)
\in\widehat X\subseteq \KK^{n+m}$
defines a point $x\in X$ by $x:=p(\overline x)$;
the points $\overline x\in\widehat X$ are called Cox coordinates of $x$.
We denote the set of all weights corresponding to a non-zero 
coordinate of $\overline x$ by
$$
W_{\overline x}
:=\{w_{ij};\;\overline x_{ij}\neq 0\}
\cup
\{u_{k};\;\overline x_{k}\neq 0\}.
$$
Moreover, let $\Cl(X,x)$ denote the local divisor
class group in $x$, i.e. the group of 
all divisor classes that are principal near $x$.

\begin{proposition}\label{Prop:FanoPicard}
Let $X$ be a $\QQ$-factorial complete normal
variety 
with complexity-one torus action
and Picard number one as constructed in \ref{Con:varX}
and set $\gamma_i:=\deg(g_i)$, $0\leq i\leq r$.
Then the following statements hold: 
\begin{itemize}
\item[(i)]
For any 
$\overline{x} \in \widehat{X}$,
the local divisor class group $\Cl(X,x)$ 
of $x := p(\overline{x})$ 
is finite and 
$\gcd(w^0; \; w\in W_{\overline x})$ 
always divides the order of the group.
\item[(ii)]
The Picard group $\Pic(X)$ is free and
the Picard index is given by  
\begin{eqnarray*}
[\Cl(X):\Pic(X)]
& = &
\mathrm{lcm}_{x \in X}(\gcd(w^0; \; w\in W_{\overline x}))\cdot|\Cl(X)^t|.
\end{eqnarray*}
In particular $|\Cl(X)^t|$ is a divisor
of $[\Cl(X):\Pic(X)]$ and we have $|\Cl(X)^t|\leq[\Cl(X):\Pic(X)]$.
\item[(iii)]
For the anticanonical class $-K_X\in\Cl(X)$ and 
its self intersection number $d_X:=(-K_X)^d$ one has
\begin{align*}
-K_X
& = \
\sum_{i=0}^r\sum_{j=1}^{n_i} w_{ij}
+\sum_{k=1}^mu_k-\sum_{i=0}^{r-2}\gamma_i,\\
\quad
d_X
& = \
\left(\sum_{i=0}^r\sum_{j=1}^{n_i} w_{ij}^0
+\sum_{k=1}^mu_k^0-\sum_{i=0}^{r-2}\gamma_i^0\right)^d
\frac{\gamma_0^0\cdots\gamma_{r-2}^0}
{\prod_{i=0}^{r} \prod_{j=1}^{n_i} w_{ij}^0
\prod_{k=1}^mu_k^0\cdot|\Cl(X)^t|}.
\end{align*}
\item[(iv)]
The variety $X$ is Fano 
if and only if the following 
inequality holds:
\begin{eqnarray*}
(r-1)\deg(g_0)^0
 \ = \  
 \sum_{i=0}^{r-2} \deg (g_i)^0
 & < &
 \sum_{i=0}^{r}\sum_{j=1}^{n_i} w_{ij}^0
+\sum_{i=1}^mu_k^0.
\end{eqnarray*}
\end{itemize}
\end{proposition}

\begin{proof}
Let $\overline x(i,j)$ resp. $\overline x(k)$
be a point in $\widehat X$ having the
$ij$-th resp. $(n+k)$-th entry one and all others zero.
With $\widehat Z:=\KK^{n+m}\setminus\{0\}$
we obtain a commutative diagram 
\[
\begin{xy}
\xymatrix{
{\widehat X} 
\
\ar@^{(-}[r]
\ar[r]
\ar[d]_{\sslash H}
&
{\widehat Z}
\ar[d]^{\sslash H}\\
X \
\ar@^{(-}[r]
\ar[r]
& Z
}
\end{xy}
\]
where the induced map embeds $X$
into a toric variety $Z$ 
such that $\Cl(X)\cong\Cl(Z)$
and $\Pic(X)\cong\Pic(Z)$ holds; 
see \cite[Corollary~III.3.1.7]{ArDeHaLa}.
By choice $\overline x(i,j)$ resp. $\overline x(k)$
is a toric fixed point.
Consequently, the Picard group $\Pic(Z)$,
and also $\Pic(X)$, is free
 \cite[Theorem VII. 2.16]{Ew}.
According to \cite[Corollary~4.9]{Ha2}
we obtain
$$
\Pic(X)
\ = \
\bigcap_{\overline x \in \widehat X}
\langle w;\; w\in W_{\overline x}\rangle
\ \cong \
\bigcap_{\overline x \in \widehat X}
\langle w^0;\; w\in W_{\overline x}\rangle,
$$
where the last equality follows from the fact
that $\Pic(X)$ is free.
This proves assertions (i) and (ii).
The remaining statements are special cases
of~\cite[Proposition~4.15 and Corollary~4.16]{Ha2}.
The self intersection number can be easily computed
by using toric intersection theory in the ambient
toric variety; compare \cite[Construction~III 3.3.4]{ArDeHaLa}.
\end{proof}

\begin{corollary}
Let $X$ be a $\QQ$-factorial
complete normal
variety with 
complexity-one torus action
and Picard number one. If $X$ is 
locally factorial,
then the divisor class group $\Cl(X)$
is free.
\end{corollary}

The following example shows that one can use
 Proposition \ref{Prop:FanoPicard}(iv)  to create
series of Fano varieties by altering the torsion part 
of the divisor class group $\Cl(X)$:

\begin{example}\label{Ex:difftorsion}
Set $l_{01}=7$, $l_{02}=1$, $l_{11}=5$ and $l_{21}=2$
as well as $w_{01}^0=1$, $w_{02}^0=3$, $w_{11}^0=2$ and 
$w_{21}^0=5$. 
According to Construction \ref{Con:R(A,n,L,m)}
these data define one single Cox ring relation 
of the form
$g_0=T_{01}^7T_{02}+T_{11}^5+T_{21}^2$.
Since we have
$$
w_{01}^0+w_{02}^0+w_{11}^0+w_{21}^0
\ = \
11
\ > \
10
\ = \
\deg(g_0)^0,
$$
one can use these data to create 
Cox rings of Fano varieties.
We provide some possible $\Cl(X)$-gradings,
given by the 
matrices $Q_i$, 
defining Del Pezzo $\KK^*$-surfaces
with fixed grading in the free part
of the divisor class group and varying
torsion part of the class group $\Cl(X)^t$:
\begin{align*}
&Q_1
\ = \
\begin{pmatrix}
1 & 3 & 2 & 5\\
\end{pmatrix},
&&\Cl(X_1)=\ZZ;
\\
&Q_2
\ = \
\begin{pmatrix}
1 & 3 & 2 & 5\\
\overline 0 & \overline 2 & \overline{1}& \overline 1
\end{pmatrix},
&&\Cl(X_2)=\ZZ\oplus\ZZ/3\ZZ;
\\
&Q_3
\ = \
\begin{pmatrix}
1 & 3 & 2 & 5\\
\overline 2 & \overline 1 & \overline 3& \overline 3
\end{pmatrix},
&&\Cl(X_3)=\ZZ\oplus\ZZ/9\ZZ;
\\
&Q_4
\ = \
\begin{pmatrix}
1 & 3 & 2 & 5\\
\overline 0 & \overline 1 & \overline{9}& \overline 6
\end{pmatrix},
&&\Cl(X_4)=\ZZ\oplus\ZZ/11\ZZ;
\\
&Q_5
\ = \
\begin{pmatrix}
1 & 3 & 2 & 5\\
\overline 0 & \overline 3 & \overline{11}& \overline 8
\end{pmatrix},
&&\Cl(X_5)=\ZZ\oplus\ZZ/13\ZZ;
\\
&Q_6
\ = \
\begin{pmatrix}
1 & 3 & 2 & 5\\
\overline 0 & \overline 7 & \overline{15}& \overline{12}
\end{pmatrix},
&&\Cl(X_6)=\ZZ\oplus\ZZ/17\ZZ.
\end{align*}
Note that in this situation not every group of the form
$\ZZ\oplus\ZZ/k\ZZ$, $k\in\NN_{>0}$, can be realized as divisor
class group.
\end{example}

In Example \ref{Ex:difftorsion} the numbers
$\ell_i:=\gcd(l_{i1},\ldots,l_{in_i})$ are pairwise coprime,
namely $\ell_0=1,\;\ell_1=2$ and $\ell_2=5$.
This allows the case $\Cl(X_1)=\ZZ$;
see \cite[Theorem~1.9]{HaHeSu}.
If the numbers $\ell_i$ are not pairwise coprime, 
then there is always non trivial torsion in the
divisor class group
as the following lemma shows.

\begin{lemma}\label{ggTPicind}
Set $\ell_i:=\gcd(l_{i1},\ldots,l_{in_i})$. 
Then all numbers $\gcd(\ell_i,\ell_j)$, where $0\leq i\neq j\leq r,$ 
divide $|\Cl(X)^t|$
and 
the Picard index $\mu$.
In particular this holds for
$\lcm_{j\neq i}(\gcd(\ell_i,\ell_j))$.
\end{lemma}

\begin{proof}
According to \cite[Theorem~1.5]{HaHe}
the divisor class group $\Cl(X)$
is isomorphic to $\ZZ^{n+m}/\mathrm{im}(P^*)$
where $P^*$ is dual to $P\colon\ZZ^{n+m}\to\ZZ^{n+m-1}$
given by a  matrix of the form
$$
P
\ = \
\begin{pmatrix}
-l_0 & l_1 & \ldots &  0 & 0\\
\vdots & & \ddots & \vdots &\\
-l_0 & 0 & \ldots & l_r & 0\\
d_0 & d_1 & \ldots & d_r & d'
\end{pmatrix},
$$ 
with $l_i=(l_{i0},\ldots,l_{in_i})$
and some integral block matrices $d_i$ and $d'$.
Consequently $|\Cl(X)^t|$ is the product of
all elementary divisors of $P$
which implies that  $\gcd(\ell_0,\ell_j)$ divides $|\Cl(X)^t|$.
By an elementary row transformation
we obtain
the analogous result for $\gcd(\ell_i,\ell_j)$
where $0\leq i,j\leq r$, $i\neq j$.
Since $|\Cl(X)^t|$ divides the Picard index $\mu$, 
the assertion follows.
\end{proof}

\begin{remark}
One can even prove that  
$\lcm_{0\leq j\leq r}(\prod_{i\neq j}\gcd(\ell_i,\ell_j))$
divides $|\Cl(X)^t|$ (see for example surface number $3$ in 
Theorem \ref{Theo:surfaces<=6}). 
\end{remark}

\section{Effective bounds}

First we consider the case $n_0 = \ldots = n_r=1$,
that means that each relation $g_{i}$
of the Cox ring $\mathcal{R}(X)$ depends
only on three variables.
Then we have $n=r+1$ and 
consequently $m=d-1$.
Furthermore, we may write $T_i$ instead of $T_{i1}$ 
and $w_i$ instead of $w_{i1}$, etc..
In this setting, 
we obtain the following bounds 
for the numbers of possible varieties~$X$
(Fano or not).

\begin{proposition}
\label{prop:Finite3Var}
For any pair $(d,\mu) \in \ZZ^2_{>0}$
there is, up to deformation equivalence, 
only a finite number of 
complete $d$-dimensional
varieties with Picard number one,
Picard index $[\Cl(X):\Pic(X)] = \mu$
and Cox ring of the form
$$ 
\KK[T_0,\ldots, T_r,S_1,\ldots, S_m]
\ / \
\langle{\alpha_{i} T_i^{l_i}
+\alpha_{i+1} T_{i+1}^{l_{i+1}}
+\alpha_{i+2} T_{i+2}^{l_{i+2}}; 
\; 0 \leq i \leq r-2}
\rangle.
$$
In this situation we have $r <\mu+\xi(\mu)-1$
where $\xi(\mu)$ denotes the number of primes smaller than $\mu$.
Moreover for $w_i^0\in\ZZ_{>0}$ and
$u_k^0\in\ZZ_{>0}$ where  $0 \leq i \leq r$, $1 \leq k \leq m$,
and the exponents $l_i$ 
one has
$$
l_i\leq\mu,
\qquad 
w_i^0 \ \leq \ \mu^r,
\qquad
u_k^0 \ \leq \ \mu.
$$
\end{proposition}

\begin{proof}
Consider the total coordinate space $\overline X\subseteq\KK^{r+1+m}$ and 
the quotient $p\colon\widehat X\to X$
as well as the points $\overline x(k)\in\widehat X$ having
the $(r+k)$-th coordinate one and all others zero. 
Set $x(k):=p(\overline x(k))$. 
Then $u_k^0$ divides the order 
of the local class group  $\Cl(X,x(k))$.
In particular we have $u_k^0\leq\mu$.\\
For each $0\leq i\leq r$ fix a point 
$\overline y(i)=(\overline y_0,\ldots,\overline y_r,0,\ldots,0)$
in $\widehat X$ such that
$\overline y_i=0$ and $\overline y_j\neq 0$
for $i\neq j$, and set $y_i:=p(\overline y(i))$. 
Then we obtain
$$
\gcd(w_j^0,j\neq i)
\; \mid \;
|\Cl(X,y(i))|.
$$
By Lemma \ref{ggTPicind} 
we have
$\lcm_{j\neq i}(\gcd(l_i,l_j))\mid |\Cl(X)^t|$.
Now consider $l_i'$ such that 
$l_i=\lcm_{j \neq i}(\gcd(l_i,l_j))\cdot l_i'$.
Then the homogeneity condition $l_iw_i^0=l_jw_j^0$ 
gives $l_i'\mid w_j^0$ for all $j\neq i$ and
consequently $l_i'\mid\gcd(w_j^0,j\neq i)$. 
Since $l_i= l_i'\cdot\lcm_{j\neq i}(\gcd(l_i,l_j))$
we can conclude $l_i\leq\mu$
by using the formula
$$
[\Cl(X):\Pic(X)]
\ = \
\lcm_{x\in X}(\gcd(w^0;\; w\in W_{\overline X}))\cdot |\Cl(X)^t|
$$
of Proposition \ref{Prop:FanoPicard}(ii).
Since the $l_i'$ are pairwise coprime
we obtain
$l_0'\cdots l_r'\mid\gamma^0$
and $l_0'\cdots l_r'\mid\mu$ 
where $\gamma^0:=\deg(g_0)^0=l_iw_i^0$.
From $l_iw_i^0=l_jw_j^0$ we deduce that
$$
l_i=l_0\frac{w_0^0}{w_i^0}=
l_0\frac{w_0^0\cdots w_{i-1}^0}{w_1^0\cdots w_i^0}
=\eta_i\cdot\frac{\gcd(w_0^0,\ldots,w_{i-1}^0)}{\gcd(w_0^0,\ldots, w_i^0)}
\leq \mu,
$$
where $1\leq \eta_i\leq \mu$. 
In particular the last fraction is smaller than $\mu$.
All in all this gives us
\begin{align*}
w_0^0
&=
\frac{w_0^0}{\gcd(w_0^0,w_1^0)}\cdot
\frac{\gcd(w_0^0,w_1^0)}{\gcd(w_0^0,w_1^0,w_2^0)}\cdot
\ldots\cdot
\frac{\gcd(w_0^0,\ldots,w_{r-2}^0)}{\gcd(w_0^0,\ldots,w_{r-1}^0)}
\cdot
\gcd(w_0^0,\ldots,w_{r-1}^0)\\
&\leq \mu^{r-1}\cdot\mu=\mu^r.
\end{align*}
Analogously we get the boundedness for all $w_i^0$.
Now let $q$ be the number of $l_i'$ that are greater than one.
Since all $l_i'$, $0\leq i\leq r$, are coprime
$q$ is bounded by $\xi(\mu)$ the number of
primes smaller than $\mu$.
To avoid the toric case we assume $l_i\neq 1$ for all $0\leq i\leq r$.
Consequently if $l_i'=1$, then there is at least one $0\leq j\leq r$
such that $\gcd(l_i,l_j)>1$. Since $\gcd(l_i,l_j)$ divides $\mu$
we get $r+1-q<\mu$ as a rough bound.
All in all we get
$
r+1
\ = \
r+1-q+q
\ < \
\mu+\xi(\mu).
$
\end{proof}

\begin{proof}[Proof of Theorem \ref{Theo:sepquot}]
Let $X$ be a variety as  in Theorem \ref{Theo:sepquot}.
Then each monomial of the Cox ring relations 
depends on only one variable,
i.e. $n_i=1$ for $0\leq i\leq r$;
for details see \cite[Theorem~1.2]{HaSu}.
Consequently
Proposition \ref{prop:Finite3Var} 
provides bounds for the discrete data
such as the non torsion parts of the weights $w_{ij}^0$ and $u_k^0$,
the exponents $l_{ij}$ and the number of Cox ring relations $r$.
Since $|\Cl(X)^t|\leq\mu$ holds, 
the number of possibilities for 
the torsion part of the grading
is also restricted
which implies the assertion.
\end{proof}

\begin{theorem}
\label{Th:FiniteIndex}
Let $X$ be a  Fano variety with complexity-one torus action 
as introduced in Construction \ref{Con:varX}.
Fix the dimension $d = \dim(X)=m+n+r$
and the Picard index 
$\mu = [\Cl(X):\Pic(X)]$.
Then the number of Cox ring relations $r$, the free part of the degree 
of the relations $\gamma^0$, the weights
$w_{ij}^0,\; u_k^0$ and the exponents $l_{ij}^0$,
where $0 \le i \le r,\;1 \le j \le n_i$ and $1\leq k\leq m$,
are bounded. In particular one obtains the following effective bounds:
We have 
$$
u_k^0 \ \le \ \mu \quad \text{for } 1 \le k \le m
\quad
\text{and}
\quad
|\Cl(X)^t|\leq\mu.
$$
Moreover, the handling of the remaining data 
can be organised in five cases,
where $\xi(x)$ denotes the number
of primes smaller than $x$.
\begin{enumerate}
\item[(i)]
Suppose that $r = 0,1$ holds. 
Then $n + m \le d+1$ holds 
and one has the bounds
$$
w_{ij}^0 \ \le \ \mu 
\quad\text{for } 0 \le i \le r \text{ and } 1 \le j \le n_i,
$$
and the Picard index is given by
$$
\mu 
\ = \ 
\mathrm{lcm}(w_{ij}^0,u_k^0; 
\;0 \le i \le r,  1 \le j \le n_i, 1 \le k \le m )
\cdot|\Cl(X)^t|.
$$
\item[(ii)]
Suppose that $r \ge 2$ and $n_0=1$ hold.
Then $r \le \mu+\xi(\mu)-1$ and $n=r+1$ and 
$m=d-1$ hold and one has 
$$
\qquad\qquad
w_{i1}^0 \ \le \ \mu^r,
\quad
l_{i1} \ \mid \ \mu
\qquad
\text{for } 0 \le i \le r,
\qquad
\gamma^0 \ \le \ \mu^{r+1},
$$
and the Picard index is given by
$$
\mu 
\ = \
\mathrm{lcm}(\gcd_i(w_{j1}^0; \; i\neq j),u_k^0;
\; 0 \le i \le r , 1 \le k \le m )
\cdot|\Cl(X)^t|.
$$
\item[(iii)]
Suppose that $r \ge 2$ and $n_0 > n_1=1$ hold.
Then we may assume $l_{11} \geq \ldots \geq l_{r1} \ge 2$,
we have $r \le\mu+ \xi(6d\mu)-1$ and
$n_0+m = d$ and the bounds
$$
w_{01}^0,\ldots,w_{0n_0}^0 \ \le \ \mu,
\qquad
l_{01},\ldots,l_{0n_0} \ \leq \ 6d\mu, 
\qquad
\gamma^0 \ < 6d\mu,
$$
$$
w_{11}^0\ < \ 2d\mu,
\quad
w_{21}^0\ < \ 3d\mu, 
\qquad
w_{i1}^0,\; 
l_{i1} \ < \  6d\mu 
\quad
\text{for } 1 \le i \le r,
$$
and the Picard index is given by
$$
\mu
\ = \
\mathrm{lcm}(w_{0j}^0, \gcd(w_{11}^0,\ldots,w_{r1}^0), u_k^0; \; 
1 \le j \le n_0, 1 \le k \le m )
\cdot|\Cl(X)^t|.
$$
\item[(iv)]
Suppose that $n_1 > n_2 = 1$ holds.
Then we may assume $l_{21} \geq \ldots \geq l_{r1} \ge 2$,
we have $r \le\mu+ \xi(2(d+1)\mu)-1$ 
and $n_0+n_1+m = d+1$ and the bounds
$$
 w_{ij}^0 \ \le \ \mu
\quad
\text{for } i=0,1 \text{ and }  1 \le j \le n_i,
\qquad
w_{21}^0 \ < \ (d+1)\mu,
$$
$$
\gamma^0,\;w_{ij}^0,l_{ij} \ < \ 2(d+1)\mu
\quad
\text{for } 0 \le i \le r ,\text{ and } 1 \le j \le n_i,
$$
and the Picard index is given by
$$
\mu
\ = \
\mathrm{lcm}(w_{ij}^0, u_k^0; 
\; 0 \le i\le 1,  1 \le j \le n_i, 1 \le k \le m)
\cdot
|\Cl(X)^t|.
$$
\item[(v)]
Suppose that $n_2 > 1$ holds
and let $s$ be the maximal number with $n_{s}>1$.
Then one may assume $l_{s+1,1} \geq \ldots \geq l_{r1} \ge 2$,
we have $s\leq d$,  $r \le\mu+ \xi((d+2)\mu)+d-1$ and 
$n_0+ \ldots + n_s+m = d+s$ and the bounds
$$
w_{ij}^0 \ \le \ \mu, 
\quad \text{for } 0 \le i \le s,
\qquad 
\gamma^0 \ < \ (d+2)\mu,
$$
$$
w_{ij}^0, l_{ij} \ <  \ (d+2)\mu
\quad \text{for } 0 \le i \le r  \text{ and } 1 \le j \le n_i,
$$
and the Picard index is given by
$$
\mu
\ = \
\mathrm{lcm}(w_{ij}^0, u_k^0; 
\; 0 \le i \le s, 1 \le j \le n_i, 1 \le k \le m)
\cdot
|\Cl(X)^t|.
$$
\end{enumerate} 
Note that assertion (i) and (ii)
do not require the Fano condition.
\end{theorem}

The remaining part of this chapter is devoted 
to the proofs of the main statements of this paper.
To prove Theorem \ref{Th:FiniteIndex}
we need the following essential lemma.

\begin{lemma}
\label{Lem:1relation}
Consider the ring
$\KK[T_{ij}; \; 0 \le i \le 2, \; 1 \le j \le n_i][S_1,\ldots,S_k]
 / 
\bangle{g}
$
where $n_0 \ge n_1 \ge n_2 \ge 1$ holds 
and let $K$ be a finitely generated abelian group
of the form $K=\ZZ\oplus K^t$ with torsion part $K^t$.
Suppose that $g$ is homogeneous 
with respect to the $K$-grading 
of $\KK[T_{ij},S_k]$ given by 
$\deg \, T_{ij} = :w_{ij}=w_{ij}^0 + w_{ij}^t \in K$
with $w_{ij}^0 \in\ZZ_{>0}$ 
and $\deg \, S_k = :u_k=u_k^0+u_k^t \in K$ with
$u_k^0 \in\ZZ_{> 0}$,
and assume 
\begin{eqnarray*}
\deg (g)^0
& < & 
\sum_{i=0}^2\sum_{j=1}^{n_i}w_{ij}^0
\ + \
\sum_{i=1}^m u_i^0.
\end{eqnarray*} 
Let $\mu \in \ZZ_{>1}$, assume 
$w_{ij}^0 \le \mu$ whenever $n_i > 1$,
$1 \le j \le n_i$ and $u_k^0 \le \mu$ 
for $1 \le k \le m$ and 
set $d := n_0+n_1+n_2+m-2$.
Depending on the shape of $g$, 
one obtains the following bounds.
\begin{enumerate}
\item[(i)]
Suppose that
$g
= 
\eta_0 T_{01}^{l_{01}}  \cdots  T_{0n_0}^{l_{0n_0}} 
+  
\eta_1 T_{11}^{l_{11}}
+
\eta_2 T_{21}^{l_{21}}$ 
with $n_0 > 1$ and coefficients 
$\eta_i \in \KK^*$ holds.
If we have $l_{11} > l_{21} \ge 2$
and $\gcd(l_{11},l_{21})\mid \mu$,
then one has
$$ 
w_{11}^0\ < \ 2d\mu,
\quad
w_{21}^0\ < \ 3d\mu, 
\quad
\quad
l_{22}, l_{21},\deg (g)^0  \ < \ 6d\mu.
$$
If we have $l_{11}=l_{21}\geq 2$, then one has
$$
l_{11} \, ,w_{11}^0\, ,l_{21}\, , w_{21}^0\, ,
\deg (g)^0\ \leq\ \mu.
$$
\item[(ii)]
Suppose that
$g
=
\eta_0 T_{01}^{l_{01}}  \cdots T_{0n_0}^{l_{0n_0}}
+
\eta_1 T_{11}^{l_{11}} \cdots  T_{1n_1}^{l_{1n_1}}
+
\eta_2 T_{21}^{l_{21}}$
with $n_1 > 1$ and coefficients 
$\eta_i \in \KK^*$ holds and 
we have $l_{21} \ge 2$.
Then one has
$$
\qquad\qquad
w_{21}^0
\ < \ 
(d+1)\mu,
\qquad\qquad
\deg (g)^0 
\ < \ 
2(d+1)\mu.
$$
\end{enumerate}
\end{lemma}

\begin{proof}
We prove~(i). Set for short
$c := (n_0+m)\mu = d\mu$.
Then, using homogeneity of $g$ 
and the assumed inequality, we obtain
$$
l_{11}w_{11}^0
\ = \ 
l_{21}w_{21}^0
\ = \ 
\deg (g)^0
\ < \ 
\sum_{i=0}^2\sum_{j=1}^{n_i}w_{ij}^0
+
\sum_{i=1}^m u_i ^0
\ \le \ 
c+w_{11}^0+w_{21}^0.
$$
First have a look at the case
$l_{11} > l_{21} \ge 2$.
Plugging this into the above inequalities, 
we arrive at 
$2 w_{11}^0 < c + w_{21}^0$ and 
$w_{21}^0 < c + w_{11}^0$.
We conclude $w_{11}^0 < 2c$ and $w_{21}^0 < 3c$.
Consequently we obtain
$$
\deg (g)^0 
\ <  \ 
c + w_{11}^0 + w_{21}^0 
\ <  \ 
6c=6d\mu.
$$
If we have $l_{11}=l_{21}$, the homogeneity condition
$l_{11}w_{11}^0=l_{21}w_{11}^0$ gives us  $w_{11}^0=w_{21}^0$.
Thus we have $\gcd(w_{11}^0,w_{21}^0)=w_{11}^0=w_{21}^0\mid\mu$
and by assumption 
$\gcd(l_{11},l_{21})=l_{21}=l_{11}\mid\mu$.
Consequently
$l_{11},w_{11}^0 ,l_{21} , w_{21}^0 ,\deg (g)^0\leq\mu$.
\\
We prove (ii).
Here we set $c := (n_0+n_1+m)\mu = (d+1)\mu$.
Then the assumed inequality gives
$$
l_{21}w_{21}^0
\ = \ 
\deg (g)^0
\ < \
\sum_{i=0}^1\sum_{j=1}^{n_i}w_{ij}^0+
\sum_{i=1}^m u_i^0+ w_{21}^0
\ \le \
c+w_{21}^0.
$$
Since we assumed $l_{21} \geq 2$, 
we can conclude $w_{21}^0 < c$.
This in turn gives us
$\deg (g)^0 < 2c$.
\end{proof}

\begin{proof}[Proof of Theorem \ref{Th:FiniteIndex}]
As before, we denote by $\overline{X} \subseteq \KK^{n+m}$
the total coordinate space and we consider the quotient
$p \colon \widehat{X} \to X$.

We first discuss the case that 
$X$ is a toric variety.
Then the Cox ring is a polynomial ring,
$\mathcal R(X) = \KK[S_1,\ldots,S_m]$.
For each $1 \le k \le m$, 
consider the point
$\overline x(k) \in \widehat{X}$
having the $k$-th coordinate one 
and all others zero and 
set $x(k) := p(\overline {x}(k))$.
Then, by~Proposition~\ref{Prop:FanoPicard},
the order of the local class group
$\Cl(X,x(k))$ is divisible by $u_k^0$.
Together with Proposition \ref{Prop:FanoPicard}(ii)
we get
$u_k^0\le \mu$ for $1 \le k \leq m$ and $|\Cl(X)^t|\le\mu $
which settles assertion~(i).

Now we treat the non-toric case,
which means $r \ge 2$.
Note that we have $n \ge 3$.
The case $n_0=1$ is done in 
Proposition~\ref{prop:Finite3Var},
which proves assertion (ii).
Hence, we are left with $n_0>1$.
For every $i$ with $n_i > 1$
and every $1 \le j \le n_i$,
there is the point $\overline{x}(i,j) \in \widehat{X}$
with $ij$-coordinate $T_{ij}$ equal 
to one and all others equal to
zero, and thus we have the point
$x(i,j) := p(\overline{x}(i,j)) \in X$.
Moreover, for every $1 \le k \le m$, we have 
the point $\overline{x}(k) \in \overline{X}$
having the $k$-coordinate $S_k$ equal to one 
and all others zero; we 
set $x(k):=p(\overline{x}(k))$.
Proposition~\ref{Prop:FanoPicard}
provides the bounds
\begin{equation}\label{ie1}
w_{ij}^0
\ \le \ 
\mu,
\quad
u_k^0
\ \le \  
\mu
\qquad
\text{for }
\quad 
n_i > 1, \,
0\leq i\leq r, \,
1 \le j \le n_i, \, 
1 \le k \le m.
\end{equation}

Let $0 \le s \le r$ be the maximal number with 
$n_{s} > 1$. Then $g_{s-2}$ is the last 
polynomial such that each of its three monomials
depends on more than one variable.
For any $t \ge s$, we have the ``cut ring'' 
\begin{eqnarray*}
R_t
& := &
\KK[T_{ij},S_k]
\ / \
\bangle{g_0, \ldots,g_{t-2}}
\end{eqnarray*}
where  
$0 \le i \le t$, $1 \le j \le n_i$, $1\le k\le m$
and the relations $g_{i}$ depend on
only three variables as soon as $i > s$ holds.
For the free part of the degree $\gamma^0$ of the relations we have 
\begin{eqnarray*}
(r-1)\gamma^0
& = & 
(t-1)\gamma^0 \ + \ (r-t)\gamma^0
\\
& = &
(t-1)\gamma^0 \ + \ l_{t+1,1}w_{t+1,1}^0 + \ldots + l_{r1}w_{r1}^0
\\
& < & 
\sum_{i=0}^r\sum_{j=1}^{n_i}w_{ij}^0
\ + \ 
\sum_{i=1}^m u_i^0
\\
& = & 
\sum_{i=0}^t \sum_{j=1}^{n_i}w_{ij}^0
\ + \ 
w_{t+1,1}^0+ \ldots + w_{r1}^0
\ + \ 
\sum_{i=1}^m u_i^0.
\end{eqnarray*}
Note that the inequality is derived from
the Fano condition of Proposition \ref{Prop:FanoPicard}(iv).
Since $l_{i1}w_{i1}^0 > w_{i1}^0$ holds in 
particular for $t+1 \le i \le r$,
we derive from this the inequality
\begin{eqnarray}\label{ie2}
\gamma^0
& < &
\frac{1}{t-1}
\left(
\sum_{i=0}^t\sum_{j=1}^{n_i}w_{ij}^0
\ + \ 
\sum_{i=1}^m u_i^0
\right).
\end{eqnarray}
To obtain the bounds in 
assertions~(iii) and~(iv),
we consider the cut ring $R_t$ 
with $t=2$ and apply
Lemma~\ref{Lem:1relation} and Proposition \ref{Prop:FanoPicard};
note that we have 
$d = n_0+n_1+n_2+m-2$ 
for the dimension $d = \dim(X)$
and that $l_{21} \ge 2$ is due 
to the fact that $X$ is non-toric.
The bounds $w_{i1}^0, l_{i1} < 6d\mu$ 
for $3 \le i \le r$ in 
assertion~(iii) follow from 
$\gamma^0 < 6d\mu$.
Similarly  $w_{ij}^0,l_{ij} < 2(d+1)\mu$ 
for $0\leq i\leq r$, $1\leq j\leq n_i$
in assertion~(iv) follow from 
$ \gamma^0 < 2(d+1) \mu$.
We still have to prove the restriction for the number of relations,
which means bounding $r$.
Recall from Lemma \ref{ggTPicind}
the definition $\ell_i:=\gcd(l_{i1},\ldots,l_{in_i})$
and set $\ell_i=\lcm_{0\leq j\neq i\leq r}(\gcd(\ell_i,\ell_j))\cdot \ell_i'$. 
Then $\ell_0',\ldots, \ell_r'$ are coprime.
For $i\geq 1$ we have $n_i=1$. 
Thus analogously to the proof of 
Proposition \ref{prop:Finite3Var} we get
$r+1=r+1-q+q\le\mu+\xi(6d\mu)$ 
where $q$ is the number of $\ell_i'$ 
that are greater than one and satisfy $n_i=1$.
For the bound in assertion~(iv)
the same argument yields
$r+1=r+1-q+q\le\mu+\xi(2(d+1)\mu)$.

To obtain the bounds in assertion~(v),
we consider the cut ring $R_t$ 
with $t=s$. 
Using $n_i=1$ for $i \ge t+1$
and applying the inequalities \ref{ie1} and \ref{ie2}, 
we can derive an upper bound for
the degree of the relation
as follows:
$$
\gamma^0
\ < \
\frac{(n_0 + \ldots + n_t + m) \mu}{t-1}
\ = \ 
\frac{(d + t) \mu}{t-1}
\ \le \ 
(d + 2) \mu.
$$
We have $w_{ij}^0l_{ij} \le \gamma^0$ 
for any $0 \le i \le r$ and any $1 \le j \le n_i$,
which implies that all $w_{ij}^0$ and $l_{ij}$ 
are bounded by $(d+2)\mu$.
Since $n_0,\ldots,n_{s-1}>1$ holds, 
the number $s$ is bounded by
$s=2s-(s-1)-1\le d$.
Consequently we obtain
$
r+1=r+1-s-q+s+q\le\mu+\xi((d+2)\mu)+d,
$
where $q$ is defined as above.

Finally, we have to express the Picard index
$\mu$ in terms of the free part of the 
weights $w_{ij}^0$, $u_k^0$
and the torsion part $\Cl(X)^t$
as claimed in the assertions.
This is a direct application of the formula of 
Proposition~\ref{Prop:FanoPicard}.
\end{proof}

\begin{proof}[Proof of Theorem
\ref{Th:asymptotic}]
Theorem \ref{Th:FiniteIndex} provides bounds
for the exponents and the number of relations
as well as for the free part of the 
weights and the torsion part of $\Cl(X)$.
Since we have $|\Cl(X)^t|\leq \mu$ the possibilities for
the torsion part of the 
weights are also restricted.
One computes that the
number $\delta(d,\mu)$ of different deformation types 
is bounded above by
$$
\mu^{\mu^2+3\mu+\xi(\mu)^2+\xi(6d\mu)+5d}
(6d\mu)^{2\mu+2\xi(6d\mu)+3d-2} 
$$
which leads to the results of
Theorem \ref{Th:asymptotic}.
\end{proof}

\begin{proof}
[Proof of Theorems \ref{Theo:surfaces<=6} and \ref{thm:3fano2}]
For fixed $d$ and $\mu$ 
Theorem  \ref{Th:FiniteIndex} bounds the number
of possible data $l_{ij}$, $w_{ij}^0$, $u_k^0$,
belonging to Fano varieties.
We identify all these constellations by a 
computer based algorithm.
Since $|\Cl(X)^t|\leq \mu$ holds,
there is only a finite number of possibilities
for the torsion part of the 
weights
that we have to check.
By this procedure we obtain the tables
of \ref{Theo:surfaces<=6} and \ref{thm:3fano2}.

We claim that any two of the listed Cox rings
do not describe varieties that are isomorphic 
to each other.
Two minimal systems of homogeneous generators
of the Cox ring contain (up to reordering) 
the same free parts of generator degrees
$w_{ij}^0$, $u_k^0\in\ZZ$.
Consequently 
they are invariant 
under isomorphy.
Furthermore the exponents $l_{ij}>1$ represent the orders
of all finite non-trivial isotropy groups
of one-codimensional orbits of the action $T$ on $X$;
see \cite[Theorem~1.3]{HaSu}.
Moreover, since none of the listed Cox rings is 
polynomial the varieties are all non-toric.
This implies that every complexity-one action
is maximal and consequently can be assigned to a maximal torus in $\Aut(X)$.
Note that $\Aut(X)$ is also acting effectively on $X$.
Since the maximal tori of $\Aut(X)$ are all conjugated the
varieties with complexity-one torus action are 
isomorphic if and only if they are 
$T$-equivariantly isomorphic.
Thus, running through the exponents $l_{ij}$
we see that any two of the varieties 
listed in Theorem \ref{Theo:surfaces<=6} are not isomorphic.

In case of Theorem \ref{thm:3fano2} there is some 
more work to do. 
There are not isomorphic threefolds
varying only in the torsion part of the 
weights,
see for example number $2$, $3$ and $4$.
In these cases,
comparing the torsion parts of the gradings
shows that it is not possible to install a
$\Cl(X)$-graded ring isomorphism between
the Cox rings of  two different threefolds.

As an example we consider
the threefolds number $2$ and $3$:
Let $D_{2}$ be a prime divisor, 
representing $\deg(T_2)\in\Cl(X)$
and let $E_1$ be a prime divisor, 
representing  $\deg(S_1)\in\Cl(X)$ .
Then  $D_{2}$ has isotropy group of order $l_{2}=3$
and  $E_1$ has infinite isotropy.
In case of threefold number $2$ the term
$D_{2}-E_1$ represents a non-trivial torsion element
whereas in case of threefold number $3$
it is the zero element in $\Cl(X)$.
Thus, these two varieties are not isomorphic.
Analogously we proceed with all other cases
to obtain finally the list of Theorem \ref{thm:3fano2}.

Finally, we apply \cite[Corollary 4.9]{Ha2} to compute
the Gorenstein index $\iota(X)$ for all listed varieties, i.e.
we have to find the smallest integer $\iota(X)$
such that  $\iota(X)\cdot K_X$ is contained in all local
divisor class groups $\Cl(X,x)$; 
see also Proposition \ref{Prop:FanoPicard}.
\end{proof}

\begin{acknowledgements}
The author would like to thank J\"urgen Hausen
and the referees
for their valuable comments, remarks and references
that improved this article a lot.
\end{acknowledgements}


\begin{thebibliography}{}
%
%
\bibitem{ArDeHaLa} 
I.~Arzhantsev, U.~Derenthal, J.~Hausen, A.~Laface:
Cox rings, arXiv:1003.4229, see also extended version on the authors' 
webpages.
%
\bibitem{Ba}
V.V.~Batyrev:
Toric Fano threefolds.
Izv. Akad. Nauk SSSR Ser. Mat. 45 (1981), no. 4,704-717, 927.
%
\bibitem{Ew}
G.~Ewald: 
Combinatorial convexity and algebraic geometry.
Grad. Texts in Math., vol.168, Springer Verlag, New york, 1996.
%
\bibitem{IaFl}
A.R.~Iano-Fletcher:
Working with weighted complete intersections.
Explicit birational geometry of $3$-folds, 101-173,
London Math. Soc. Lecture Note Ser., 281,
Cambridge Univ. Press, Cambridge, 2000.
%
\bibitem{Ha2}
J.~Hausen:
Cox rings and combinatorics II.
Mosc. Math. J. 8 (2008), no.~4,  711--757.
%
\bibitem{Ha3}
J.~Hausen:
Three lectures on Cox rings.
arXiv:1106.0854 (2011)
%
\bibitem{HaHe} 
J.~Hausen, E.~Herppich:
Factorially graded rings of complexity one.
arXiv: 1005.4194v1 (2010).
%
\bibitem{HaHeSu}
J.~Hausen, E.~Herppich, H.~S\"u{\ss}:
Multigraded factorial rings and Fano varieties
with torus action.
Documenta Math. 16 (2011), 71--109.
%
\bibitem{HaSu}
J.~Hausen, H.~S\"u{\ss}:
The Cox ring of an algebraic variety 
with torus action.
Advances Math. 225 (2010), 977--1012.
%
\bibitem{Is}
V.A. Iskovskih: Fano threefolds. II.
 Izv. Akad. Nauk SSSR Ser.Mat.42 (1978), no. 3, 506-549.
%
\bibitem{Ka}
A.~Kasprzyk:
Bounds on fake weighted projective spaces.
Kodai Math. J. 32 (2009), no.2, 197-208.
%
\bibitem{Ka2}
A.~Kasprzyk:
Canonical toric Fano threefolds.
Canad. J. Math. 62 (2010), no.6, 1293-1309.
%
\bibitem{KaKrNi}
A.~Kasprzyk, M.~Kreuzer, B.~Nill:
On the combinatorial classification 
of toric log del Pezzo surfaces.
LMS J. Comput. Math. 13 (2010), 33-46.
%
\bibitem{MoMu}
S.~Mori, S.~Mukai:
Classification of Fano $3$-folds with $B_2\geq 2$.
Manuscripta Math. 36 (1981/82), no.2, 147-162.
%
\bibitem{Sum}
H.~Sumihiro: 
Equivariant completion. 
J. Math. Kyoto Univ. 14 (1974), 1-28.
%
\bibitem{Su}
H.~S\"u{\ss}:
Canonical divisors on $T$-varieties.
Preprint, arXiv:0811.0626v1.
\end{thebibliography}
\end{document}